\documentclass{amsart}
\usepackage{color}

 \newtheorem{thm}{Theorem}[section]
 \newtheorem{cor}[thm]{Corollary}

 \theoremstyle{definition}
 \newtheorem{defn}[thm]{Definition}
 \theoremstyle{remark}
 \newtheorem{rem}[thm]{Remark}
  \theoremstyle{Example}
 
 \numberwithin{equation}{section}

\begin{document}
\address{Faculty of Mathematical sciences, The University of Guilan}
\email{ahmadi@guilan.ac.ir, maliheh@phd.guilan.ac.ir}



\title[]
{Characterization of  Entropy for Spacing shifts}
\author[D. Ahmadi, M. Dabbaghian]{Dawoud Ahmadi Dastjerdi and Maliheh Dabbaghian Amiri}
 \maketitle

 \begin{abstract}
Suppose $P\subseteq \mathbb{N}$ and let $(\Sigma_P,\,\sigma_P)$ be the space of a spacing shift. We show that if entropy $h_{\sigma_P}=0$ then  $(\Sigma_P,\,\sigma_P)$ is proximal. Also   $h_{\sigma_P}=0$ if and only if $P=\mathbb N\setminus E$ where $E$ is an intersective set. Moreover, we show that
 $h_{\sigma_P}>0$ implies that $P$ is a $\Delta^*$ set; and by giving a class of examples, we show that this is not a sufficient condition. Then there is enough results to solve  question 5 given in [J. Banks et al.,
\textit{Dynamics of Spacing Shifts}, Discrete Contin. Dyn. Syst.,
to appear.].
 \end{abstract}
\section*{Introduction and Definitions}
In this paper we  give a characterization of entropy of a spacing shifts by the combinatorial property of the set $P\subseteq \mathbb{N}$ which defines a spacing shift.
A detailed study for spacing shifts can be found in \cite{sp}, so we here only consider the basic definitions and notions needed for our task.

A topological dynamical system (TDS) is a pair $(X,\, T)$ such that
$X$ is a compact metric space and $T$ is  a continuous surjective
self map.
The \emph{orbit closure} of a point $x$ in $(X,\,T)$ is the set
$\overline{\mathcal{O}}(x)=\overline{\{T^n(x): n\in\mathbb N\}}$.
A system $(X,\,T)$
is \emph{transitive} if it has a point $x$ such that
$\overline{\mathcal{O}}(x)=X$.
 Also a point
$x$ is \emph{recurrent} if for every neighborhood $U$ of $x$
there exists $n\neq 0$ such that $T^n(x)\in U$. We let $N(x,\,U)=\{n\in\mathbb N:
T^n(x)\in U \}$ and $N(U,V)=\{n\in\mathbb N: T^n(U)\cap V\neq
\emptyset\}$ where $U$ and $V$ are open sets.

Let $x_1,\,x_2\in
X$. One says that $(x_1,\, x_2)\in X\times X$ is a \emph{proximal
pair} if $$\liminf_{n\rightarrow \infty}d(T^n(x),\,T^n(y))=0;$$ and a
TDS is called \emph{proximal} if all $(x_1,\,
x_2)\in X\times X$ are proximal pairs.

Let $A=\{a_n\}_{n \in \Bbb N}$ be an increasing sequence of natural numbers.
Then $s=a_{i_1}+a_{i_2}+...+a_{i_n}, \ \ i_j<i_{j+1}$ is called a
\textit{partial finite sum} of $A$. The \textit{finite sums} of
$A$ denoted by $ FS(A)$ is the set of all partial finite sums. A
set $F \subset \Bbb N$ is called \textit{$IP$-set} if it contains
the finite sums of some sequence of natural numbers. Let $\mathcal{IP}$ be the set of all $IP$-sets.

A set $D \subset \Bbb N$ is called \textit{$\Delta$-set} if
there exists an increasing sequence of natural numbers $S=(s_n)_{n\in\Bbb N}$
such that the difference set $\Delta(S)=\{s_i-s_j: \ i>j\}
\subset D$. Denote by $\bf\Delta$ the set of all $\Delta$-sets.
Any $IP$-set is a $\Delta$-set; for let
$S=\{a_1,a_1+a_2,a_1+a_2+a_3,...\}$.

A collection $\mathcal{F}$ of non-empty subsets  of $\mathbb{N}$
is called a \emph{family} if it is hereditary upward: if
$F\in\mathcal{F}$ and $F\subset F'$, then $F'\in\mathcal{F}$. The
dual family ${\mathcal F}^*$, is defined to be all subsets of $
\Bbb N$
 that meets all sets in ${\mathcal F}$. That is
$${\mathcal F}^*= \{G\subset \mathbb N : \ \ G\cap F\neq
\emptyset, \ \forall F \in {\mathcal F} \}.$$ Hence
$\mathcal{IP}^*$ and ${\bf\Delta}^*$ are the dual family of
$\mathcal{IP}$ and ${\bf\Delta}$ respectively.

The notions for a subset of natural numbers such as $\Delta$ or $IP$ are  structural notions. For instance, an $IP$-set is more structured than a $\Delta$-set. Other structures are also
 defined \cite{hind}, \cite{Idem}.  There are also notions for largeness which are defined by means of different densities on  subsets of natural numbers. See \cite{hind}, \cite{mul} for a rather complete treatment for both of these notions. Let $A\subseteq\mathbb{N}$. Then
$$\overline{d}(A)=\limsup_{n\rightarrow\infty}{\frac{|A\cap \{1,\cdots,\,n\}|}{n}}$$ is called the \emph{upper density} of $A$. Also the \emph{lower density} is defined as
$$\underline{d}(A)=\liminf_{n\rightarrow\infty}{\frac{|A\cap \{1,\cdots,\,n\}|}{n}}.$$
When $\overline{d}(A)=\underline{d}(A)$ then it is called the \emph{density} of $A$ and is denoted by $d(A)$. The \emph{upper Banach density} of $A$ is denoted by $d^*(A)$ and is defined as
$$d^*(A)=\limsup_{N_i-M_i\rightarrow\infty} \frac{{|A\cap \{M_i,\,M_i+1,\,\cdots,\,N_i\}|}}{N_i-M_i+1}.$$

When there is $k\in\mathbb{N}$ such that all the intervals in
$\mathbb{N}\setminus A$ have length less than $k$, then $A$ is
called \emph{syndetic}. The length of the largest of such
intervals will be called the \emph{gap} of $A$. Clearly,
$\underline{d}(A)>0$ for any syndetic set $A$. The dual of
syndetic sets are \emph{thick} sets; a set is thick if and only
if $d^*(A)=1$. We say $A$ is \emph{thickly syndetic} if for every
$N$ the positions where consecutive elements of length $N$ begins
form a syndetic set.

Note that $\Delta^*$-sets are highly structured and are syndetic \cite{Idem}. Another of such large and structured subsets of $\mathbb{N}$ are \emph{Bohr sets}. We say that a subset $A\subset \mathbb N$ is a Bohr set if
there exist $m\in \mathbb N$, $\alpha \in  \mathbb{T}=\{z\in\mathbb{C}:\ |z|=1\}$ and open
set $U\subset {\mathbb T}^m$ such that $$\{n\in\mathbb N:\
n\alpha\in U\}$$ is in $A$. In particular, every $k\mathbb N$ is a Bohr set.

\begin{defn}\label{spacing}
 For any set $P \subset \Bbb N$
 define a \emph{spacing shift} to be the  subshift
$$\Sigma_P =\{s \in \Sigma : \  s_i=s_j=1  \ \Rightarrow |i-j| \in P \cup\{0\}\}.$$
\end{defn}
For any $y\in\Sigma_P$ we associate a set $A_y=\{i: y_i=1\}$.  it is clear that $A_y-A_y\subset P$.
Therefore, notions of largeness and structure  for $A_y$ gives the same notions for incidence of $1$'s for $y$. That is we set
$$d(y):=d(A_y)=\lim_{n\rightarrow \infty}  {\frac{\sum_{1}^{n}y_i}{n}}= \lim_{n\rightarrow \infty}  {\frac{|A_y\cap \{1,\cdots,n\}|}{n}}.$$
Similarly, $\overline{d}(y)$, $\underline{d}(y)$ and $d^*(y)$ can be defined.

By Definition \ref{spacing},  it is clear that $A_y-A_y\subset P$.

\subsection*{Acknowledgements} We would like to thank Maryam Hosseini for her fruitful discussions.
\section{Zero Entropy Gives Proximality}\label{sect}
The following questions arises in \cite[Question 5]{sp}.

``Is there $P$ such that $\mathbb N\setminus P$ does not contain
$IP$-set but $\Sigma_P$ is proximal? What about positive
topological entropy? Are these two properties (i.e proximality
and zero entropy) essentially different in the context of spacing
shifts? "

We give positive answer to the first question but we will show
that if $\mathbb N\setminus P$  contains  $\Delta$-set (and hence $IP$-set), then the
entropy is zero. Also we will show that zero entropy  in spacing
shifts implies proximality.

For
any $x, y \in \Sigma_P$ let
$$F_{xy}(t) = \liminf_{n\rightarrow\infty} {{{1}\over {n}} |\{ 0 \leq m \leq n-1: d(\sigma^m(x), \sigma^m(y)) < t\}|}.$$

\begin{rem}
In \cite{sp} the authors show
 that if there are $x,\, y \in \Sigma_P$, $t>0$
such that $F_{xy}(t)<1$ then $h_{\sigma_P}>0$ . If such
$x, y$ and $t$ exist, then there is some $y'\in \Sigma_P$ such
that $\overline{d}(y')>0$. Because let $t={2^{-l}}$ then there
exists an increasing sequence $\{q_i\}_{i=1}^{\infty}$ and
$\epsilon>0$ such that either $|\{0\leq j\leq q_i: x_{j}\neq
0\}|> {{q_i\epsilon}\over{l+1}}$ or $|\{0\leq j\leq t_i:
y_{j}\neq 0\}|> {{q_i\epsilon}\over{l+1}}$. Hence
$\overline{d}(x)$ or $\overline{d}(y)$ is positive.
\end{rem}

In \cite[Lemma 3.5]{sp}, it has been proved that if $\mathbb N
\setminus P$ contains an $IP$-set then $d(y)=0$, for $y\in
\Sigma_P$. We give a stronger result with a simpler proof.

\begin{thm}
If $\mathbb N\setminus P$ contains a $\Delta$-set then $d^*(y)=0$
for all $y\in\Sigma_P$.
\end{thm}
\begin{proof}
If $y\in\Sigma_P$, then $A_y-A_y\subset P$. But if there is $y$ such that $d^*(y)>0$
then $A_y-A_y$ is a $\Delta^*$-set \cite{delta} and $\mathbb
N\setminus P$ cannot have a $\Delta$-set.
\end{proof}

The following result is a reformulation of two results in \cite{sp}.
\begin{thm}\label{0den0Ent}
If for all $y\in \Sigma_P$, $d(y)=0$, then
\begin{enumerate}
  \item $h_{\sigma_P}=0$,
  \item $\sigma_P$ is proximal.
\end{enumerate}
\end{thm}
\begin{proof}
 (1) and (2) are proved in \cite[Theorem
3.6]{sp} and \cite[Theorem 3.11]{sp} respectively for the case
when $\mathbb N\setminus P$ contains an $IP$-set. The proof
of these theorems are based on the fact that if $\mathbb
N\setminus P$ contains an $IP$-set then $d(y)=0$, for any $y\in
\Sigma_P$. Then this last result will lead to the both conclusions.
\end{proof}
Again the proof of this Theorem is a minor alteration of in the proof of \cite[Theorem 3.18]{sp}.
\begin{thm} \label{entropy}
There exists some $y\in \Sigma_P$ with $d^*(y)>0$ if and only if
$h_{\sigma_P}>0$.
\end{thm}
\begin{proof}
First suppose there exists a point $y\in \Sigma_P$ such that
$d^*(y)>0$, so for some $l$ there exist two increasing sequences
$\{M_i\}_{i=1}^{\infty}$, $\{N_i\}_{i=1}^{\infty}$  and
$\gamma>0$ such that
 $$|\{M_i\leq j\leq N_i:  \  y_{[j,\,j+l]}\neq 0^{l+1}\}|\geq (N_i-M_i)\gamma.$$
 So
 $$|\{M_i\leq j\leq N_i:  \ y_j\neq 0\}|\geq {{(N_i-M_i)\gamma}\over{l+1}}.$$
Then by definition we have
$$h_{\sigma_P}\geq \lim_{N_i-M_i\rightarrow \infty}
{{1}\over{N_i-M_i}} \log(2^{{(N_i-M_i)\gamma}\over{l+1}})>0.$$
Conversely, if for any $y\in\Sigma_P$, $d^*(y)=0$ then $ d(y)=0$ and the proof follows from Theorem \ref{0den0Ent}.
\end{proof}
An immediate consequence of the above theorem is that if $P$ is not $\Delta^*$, then $h_{\sigma_P}=0$. In particular, this sorts out the second question.

By Theorem \ref{0den0Ent},  if $h_{\sigma_P}>0$, then there is a $y\in\Sigma_P$ such that $d(y)>0$. Combining this with the results of the above Theorem we have:
\begin{cor}
There is a point $y\in\Sigma_P$ with $d(y)>0$ if and only if for some $y'$, $d^*(y')>0$.
\end{cor}

 The following gives an answer to the third question. Moreover, this result and the fact that when $P$ misses an $IP$-set then it is not $\Delta^*$ and so has zero entropy is an answer for the first question as well.
\begin{thm}\label{proximal}
If  $h_{\sigma_P}=0$ then $\Sigma_P$ is proximal.

\end{thm}
\begin{proof}
Suppose $h_{\sigma_P}=0$. Then by Theorem \ref{entropy}, for any $y\in \Sigma_P$ we have
$d^*(y)=0$ which implies that $d(\{i: y_i=0 \})=1$. Hence for any two
points $x,\,y\in \Sigma_P$, $d(\{i: x_i=0 \}\cap \{i: y_i=0\})=1$
and this in turn implies that $\Sigma_P$ is proximal.

\end{proof}
\subsection{A necessary condition for transitivity}
Still there is not a characterization for $P$ to have $\Sigma_P$ transitive. This also has been put as a question in  \cite[Question 1]{sp}. A necessity is the following.
\begin{thm}\label{transitive}
Suppose $\Sigma_P$ is transitive. Then $P$ is an $IP-IP$ set.
\end{thm}
\begin{proof}
For any TDS such as $(X,\,T)$, the return times of a recurrence
point $x$ to any non-empty open set $U$, that is, $N(x,U)=\{n\in
\mathbb N: T^n(x)\in U\}$ is an $IP$-set \cite[Theorem
2.17]{Fur}. Now let $y$ be a transitive point. Then $y$ is a
recurrence point and $N(y,\,[1])$ is an $IP$-set. But
$N(y,\,[1])=\{y_i:\ y_i=1\}=A_y$ and so $A_y-A_y\subset P$ and as
a result $P$ is an $IP-IP$ set.
\end{proof}

An application of the above theorem is that any thick subset of natural numbers is an $IP-IP$ set. This is because $\Sigma_P$ is weak mixing if and only if $P$ is thick and if a TDS is weak mixing, then it is  transitive, in fact, totally transitive: $(\Sigma_P,\,\sigma^n)$ is transitive for all $n=0,\,1,\,\ldots$.

It is not hard to see that for any infinite subset of $\mathbb{N}$ such as $A$, $P=FS(A)-FS(A)$ is a transitive system. On the other hand, let $k\geq 3$, $p_2>p_1$ and $p_2-p_1\not=kn$ for any $n\in\mathbb{N}$. Now if $P=k\mathbb{N}\cup\{p_1,\, p_2\}$, then $\Sigma_P$ is not transitive, however it is clearly $IP-IP$ set. Because it contains an $IP-IP$ set such as $k\mathbb{N}$.

By now we understand that this is the structure in $P$ and not density which gives interesting dynamics to our spacing shifts systems. For instance, if $P$ is not a $\Delta$-set-set, then
 for all $y\in \Sigma_P$, $\sum_{i=1}^{\infty}y_i<\infty$. This gives a very simple dynamics to $\Sigma_P$. In fact, it is an equicontinuous system where any point will be attracted to $0^\infty$ eventually. We may choose $P$ to have high density. As an example, for any $\epsilon>0$ let $\frac{1}{k}<\epsilon$ and set $P=\mathbb N\setminus k\mathbb N$. Then $
d(P)\geq 1-\epsilon$ and since  $k\mathbb N$ is a $\Delta^*$-set
$P$ does not contain any $\Delta$-set.
\section{Combinatorial Characterization for Zero Entropy}
In  section \ref{sect}, we showed that $P$ must be at least $\Delta^*$ set, that is a highly structured and large set to have positive entropy. Here we show that even if $P$ is a $\Delta^*$ set, it is not  guaranteed that $h_{\sigma_P}>0$.

One calls $E\subset\mathbb{N}$ a \emph{density intersective} set
if for any $A\subset\mathbb{N}$ with positive upper Banach
density, $E\cap(A-A)\neq\emptyset$. For instance, any $IP$-set is
a density intersective set. In fact, if $R\subset\mathbb{N}$ is
an $IP$-set and $p(\cdot)$ is a polynomial such that
$p(\mathbb{N})\subset\mathbb{N}$, then $E=\{p(n):\ n\in R\}$ is
a  density intersective set \cite{fer}.
\begin{thm}

$h_{\sigma_P}=0$ if and only if $P=\mathbb{N}\setminus E$ where
$E$ is a density intersective set.
\end{thm}
\begin{proof}
Suppose $h_{\sigma_P}=0$. If $E=\mathbb{N}\setminus P$ is not
density  intersective, then there must be a set A with positive
upper Banach density such that $A-A\subseteq P$. Choose
$y\in\Pi_{i=0}^\infty \{0,\,1\}$ such that $y_i=1$ if and only if
$i\in A$. Then $y\in\Sigma_P$ and $A=A_y$. But this is absurd by
Theorem \ref{entropy}.

For the other side, if $E$ is density intersective, then $P$ does
not contain any $A-A$ where $A$ is as above. Therefore, for all
$y\in\Sigma_P$, $d^*(y)=0$ which implies $h_{\sigma_P}=0$.
\end{proof}

It is an easy exercise to show that $\{n^2:\ n\in\mathbb{N}\}$
does not contain any $\Delta$-set. So $P=\mathbb{N}\setminus E$
is a $\Delta^*$ set and by the above theorem, $h_{\sigma_P}=0$.

\subsection{Positive entropy  with no non-zero periodic points}
Any spacing shift has  $0^\infty$ as its  periodic point. But a spacing shift has a non-zero periodic point of
period $k$ if and only if $P$ contains $k\mathbb N$ \cite[Lemma 2.6]{sp}.
This implies there is a point $y$ with $d(y)\geq\frac{1}{k}$ and so by Theorem \ref{entropy} we have positive entropy.
\begin{thm}
There is $P$ such that $\Sigma_P$ has positive entropy with no non-zero periodic points.
\end{thm}
\begin{proof}
A theorem of K$\check{\rm r}\acute{\rm i}\breve{\rm z}$
\cite{ruzsa} states
 that
there is a set $A$ with positive upper Banach density whose difference
set contains no Bohr set. So let $y=\{y_i\}_{i\in\mathbb{N}}$ be defined by $y_i=1$ if $i\in A$ and zero otherwise. Set  $P=A-A$. Then $y\in\Sigma_P$, $A_y=A$ and $\overline{d}(y)=\overline{d}(A)>0$. Therefore, $h_{\sigma_P}>0$ and since $P$ does not contain any Bohr set  it does not contain any $k\mathbb{N}$ and the proof is complete.
\end{proof}

\end{document}